\documentclass[12pt]{amsart}
\usepackage{amscd, amsfonts, amssymb, amsthm, mathrsfs}
\usepackage{parskip, fullpage, verbatim}
\usepackage[all]{xypic}
\usepackage{ctable}
\usepackage{longtable}
\usepackage{array}
\usepackage{aliascnt}
\usepackage[colorlinks, breaklinks, linkcolor=blue]{hyperref} 
\usepackage{breakurl}
\usepackage{booktabs}
\usepackage{wasysym}

\newcommand\CA{{\mathscr A}} 
\newcommand\CB{{\mathscr B}}
\newcommand\CC{{\mathscr C}} 
\renewcommand\CD{{\mathscr D}}
\newcommand\CE{{\mathscr E}}

\newcommand\CIF{{\mathcal {IF}}}

\newcommand\CDF{{\mathcal {DF}}} 
 
\newcommand\CSF{{\mathcal {SF}}} 
\newcommand\CAF{{\mathcal {AF}}}

\newcommand\BBK{{\mathbb K}}
\newcommand\BBQ{{\mathbb Q}}

\newcommand\BBZ{{\mathbb Z}}

\newcommand\Der{\operatorname{Der}}
\newcommand\pdeg{\operatorname{pdeg}}

\numberwithin{equation}{section}

\theoremstyle{plain}

\newtheorem{theorem}[equation]{Theorem}
\newtheorem{conjecture}[equation]{Conjecture}

\newtheorem{proposition}[equation]{Proposition}
\theoremstyle{definition}
\newtheorem{definition}[equation]{Definition}
\newtheorem{remark}[equation]{Remark}
\newtheorem{example}[equation]{Example}

\begin{document}

\subjclass[2010]{20F55, 52C35, 14N20, 32S22, 51D20, 51F15}

\title[Some remarks on free arrangements]{Some remarks on free arrangements}


\author[T. Hoge]{Torsten Hoge}
\address
{Institut f\"ur Algebra, Zahlentheorie und Diskrete Mathematik,
Fakult\"at f\"ur Mathematik und Physik,
Leibniz Universit\"at Hannover,
Welfengarten 1,
30167 Hannover, Germany}

\author[G. R\"ohrle]{Gerhard R\"ohrle}
\address
{Fakult\"at f\"ur Mathematik,
Ruhr-Universit\"at Bochum,
D-44780 Bochum, Germany}
\email{gerhard.roehrle@rub.de}

\keywords{Hyperplane arrangements, 
Terao's Conjecture, inductively free arrangements,  
divisionally free arrangements, additionally free arrangements,
stair-free arrangements}

\begin{abstract}
We exhibit a particular free subarrangement of 
a certain restriction of the Weyl arrangement of type $E_7$ 
and use it to give an affirmative answer to a recent
conjecture by  T.~Abe on the nature of additionally free and stair-free arrangements.
\end{abstract}

\maketitle
\allowdisplaybreaks


\section{Introduction}
The interplay between algebraic and combinatorial structures of 
hyperplane arrangements has been a driving force in the study of the subject 
for a long time. At the very heart of these investigations lies 
Terao's Conjecture \ref{conj:terao} which asserts that the 
algebraic property of freeness of an arrangement 
is determined by purely combinatorial data.

\begin{conjecture}
[{\cite[Conj.~4.138]{orlikterao:arrangements}}]
\label{conj:terao}
For a fixed field, 
freeness of the arrangement $\CA$ only depends on its 
lattice  $L(\CA)$, i.e.\ is combinatorial.
\end{conjecture}

In his recent papers \cite{abe:deletion} and \cite{abe:sf}, T.~Abe shows that
all free arrangements that obey Terao's Addition-Deletion Theorem \ref{thm:add-del}
are indeed combinatorial.
In \cite{abe:sf}, he introduced 
a new class of free arrangements, so called
\emph{stair-free} arrangements $\CSF$
(Definition \ref{def:sf}). 
Its significance lies in the fact that 
Terao's Conjecture \ref{conj:terao} is still valid within 
$\CSF$ (\cite[Thm.~4.3]{abe:sf}). To date this is the largest known class of free arrangements 
with this property. 
This class  encompasses the 
class of \emph{divisionally free} arrangements $\CDF$
(Definition \ref{def:divfree})
and the class of \emph{additionally free} arrangements $\CAF$
(Definition \ref{def:addfree}), \cite[Thm.~4.3]{abe:sf}.

The class of divisionally free arrangements $\CDF$ in turn contains the class of 
inductively free arrangements $\CIF$ (Definition \ref{def:indfree}),
cf.~\cite[Thm.~1.6]{abe:divfree}.
The following confirms a conjecture of Abe, 
\cite[Conj.\ 4.4]{abe:sf}, which resolves the containment 
relations among these classes of free arrangements.

\begin{theorem}
\label{thm:abe-conj4.4}
With the notation from above, we have 
\begin{itemize}
\item[(i)] $\CIF \subsetneq \CAF$;
\item[(ii)] $\CDF \not\supset \CAF$;
\item[(iii)] $\CDF \cup \CAF \subsetneq \CSF$.
\end{itemize}
\end{theorem}

In \S \ref{sec:proofs} we exhibit a subarrangement $\CD$ of the 
rank $5$ restriction of type $(E_7,A_1^2)$ of the Weyl arrangement of type $E_7$
which is additionally free but not inductively free and which at the same time is not 
divisionally free; so   
parts (i) and (ii) of Theorem \ref{thm:abe-conj4.4} follow.
It turns out that $\CD$ is the restriction of a subarrangement 
$\CB$ of the Weyl arrangement of type $E_7$ which also shares 
these features. 
These two arrangements are the only instances known to us with these properties.
Each of $\CB$ and $\CD$ is obtained from an 
inductively free arrangement by deleting a single hyperplane.
Moreover, $\CD$ is also crucially involved in our construction of an example 
in $\CSF\setminus (\CDF \cup \CAF)$ which gives part (iii) of 
Theorem \ref{thm:abe-conj4.4}.
It is quite remarkable that a particular subarrangement of
a restriction of the Weyl arrangement of type $E_7$ provides the basis for all the statements in 
Theorem \ref{thm:abe-conj4.4}.
  
Free arrangements are compatible with 
the product construction for arrangements, \cite[Prop.~4.28]{orlikterao:arrangements}.
It is easy to show that this is also the case for Abe's new classes 
$\CAF$ and $\CSF$, see Proposition \ref{prop:product-sffree}.

In addition, we show that $\CAF$ is not closed 
under taking restrictions, see \S \ref{sec:C}.
In turn $\CDF$ is not closed  
under taking localizations, see \cite[Ex.~2.16]{roehrle:divfree}.
Consequently, the larger class $\CSF$ is not closed under these operations either.

In our final section we address another conjecture of Abe, \cite[Conj.\ 3.5(2)]{abe:sf},
which states that if the characteristic polynomials $\chi( \CA, t)$ and $\chi( \CA', t)$
of $\CA$ and a deletion $\CA'$ of $\CA$
factor over $\BBZ$ and share all but one root, then both $\CA$ and $\CA'$ are free.
While this is true in dimension at most $3$, thanks to \cite[Thm.~1.1]{abe:divfree}, 
in Example \ref{ex:divnotfree}, 
we give a counterexample to this conjecture in dimension $4$.
Specifically, we present a triple of arrangements $(\CA, \CA', \CA'')$ with the property that 
none of its members is free but each of their characteristic polynomials factors over $\BBZ$ and  
$\chi(\CA'', t)$ divides both $\chi(\CA', t)$ and $\chi( \CA, t)$.
We end with a general construction for examples of this kind in  Example \ref{ex:divnotfree2}. 

\section{Preliminaries}
\label{ssect:recoll}

\subsection{Hyperplane Arrangements}
\label{ssect:hyper}

Let $\BBK$ be a field and let $V = \BBK^\ell$.
By a hyperplane arrangement in $V$ we mean a finite set $\CA$ of
hyperplanes in $V$. Such an arrangement is denoted $(\CA,V)$ 
or simply $\CA$. If $\dim V = \ell$
we call $\CA$ an $\ell$-arrangement.
The number of elements in $\CA$ is given by $|\CA|$.
The empty $\ell$-arrangement 
is denoted by $\Phi_\ell$.

By $L(\CA)$ we denote the set of all nonempty intersections of elements of $\CA$,
\cite[Def.~1.12]{orlikterao:arrangements}.
For $X \in L(\CA)$, we have two associated arrangements, 
firstly the subarrangement 
$\CA_X :=\{H \in \CA \mid X \subseteq H\} \subseteq \CA$
of $\CA$ and secondly, 
the \emph{restriction of $\CA$ to $X$}, $(\CA^X,X)$, where 
$\CA^X := \{ X \cap H \mid H \in \CA \setminus \CA_X\}$,
\cite[Def.~1.13]{orlikterao:arrangements}.
Note that $V$ belongs to $L(\CA)$
as the intersection of the empty 
collection of hyperplanes and $\CA^V = \CA$. 

If $0 \in H$ for each $H$ in $\CA$, then 
$\CA$ is called \emph{central}.
We only consider central arrangements.

Let $H \in \CA$ (for $\CA \neq \Phi_\ell$) and
define $\CA' := \CA \setminus\{ H\}$,
and $\CA'' := \CA^{H}$.
Then $(\CA, \CA', \CA'')$ is a \emph{triple} of arrangements, 
\cite[Def.~1.14]{orlikterao:arrangements}.

The \emph{product}
$\CA = (\CA_1 \times \CA_2, V_1 \oplus V_2)$ 
of two arrangements $(\CA_1, V_1), (\CA_2, V_2)$
is defined by
\begin{equation*}
\label{eq:product}
\CA := \CA_1 \times \CA_2 = \{H_1 \oplus V_2 \mid H_1 \in \CA_1\} \cup 
\{V_1 \oplus H_2 \mid H_2 \in \CA_2\},
\end{equation*}
see \cite[Def.~2.13]{orlikterao:arrangements}.
In particular, $|\CA| = |\CA_1| + |\CA_2|$. 
For $H =  H_1 \oplus V_2 \in \CA$, we have 
\begin{equation}
\label{eq:restrproduct}
\CA^H = \CA_1^{H_1} \times \CA_2.
\end{equation}

The \emph{characteristic polynomial} 
$\chi(\CA,t) \in \BBZ[t]$ 
of $\CA$ is defined by 
\[
\chi(\CA,t) := \sum_{X \in L(\CA)} \mu(X)t^{\dim X},
\]
where $\mu$ is the M\"obius function of $L(\CA)$, 
see \cite[Def.\ 2.52]{orlikterao:arrangements}.

If $\CA = \CA_1 \times \CA_2$ is a product, then, 
thanks to \cite[Lem.\ 2.50]{orlikterao:arrangements}, 
\begin{equation}
\label{eq:chiproduct}
\chi(\CA,t) = \chi(\CA_1,t) \cdot  \chi(\CA_2,t). 
\end{equation}

\subsection{Free Arrangements}
\label{ssect:free}
Let $S = S(V^*)$ be the symmetric algebra of the dual space $V^*$ of $V$.
If $\CA$ is an arrangement in $V$,
then for every $H \in \CA$ we may fix $\alpha_H \in V^*$ with
$H = \ker(\alpha_H)$.
We call $Q(\CA) := \prod_{H \in \CA} \alpha_H \in S$
the \emph{defining polynomial} of $\CA$.

The \emph{module of $\CA$-derivations} is the $S$-submodule of $\Der(S)$,
the $S$-module of $\BBK$-derivations of $S$, 
defined by 
\[
D(\CA) := \{\theta \in \Der(S) \mid \theta(Q(\CA)) \in Q(\CA) S\}.
\]
The arrangement $\CA$ is said to be \emph{free} 
if $D(\CA)$ is a free $S$-module.

If $\CA$ is a free $\ell$-arrangement, 
then $D(\CA)$ admits an $S$-basis of $\ell$ 
homogeneous derivations $\theta_1, \ldots, \theta_\ell$, 
by \cite[Prop.~4.18]{orlikterao:arrangements}.
While such a homogeneous $S$-basis of $D(\CA)$ need not be unique, 
the multiset
consisting of the polynomial degrees of the $\theta_i$
is unique. They are called the \emph{exponents} of 
the free arrangement $\CA$ and are denoted by
$\exp\CA := \{\pdeg \theta_1, \ldots, \pdeg \theta_\ell\}$.

Terao's basic \emph{Addition-Deletion Theorem} 
plays a key role in the study of free arrangements.

\begin{theorem}
[{\cite{terao:freeI}, \cite[Thm.\ 4.51]{orlikterao:arrangements}}]
\label{thm:add-del}
Suppose $\CA \neq \Phi_\ell$ and
let $(\CA, \CA', \CA'')$ be a triple of arrangements. Then any 
two of the following statements imply the third:
\begin{itemize}
\item[(i)] $\CA$ is free with $\exp\CA = \{ b_1, \ldots , b_{\ell -1}, b_\ell\}$;
\item[(ii)] $\CA'$ is free with $\exp\CA' = \{ b_1, \ldots , b_{\ell -1}, b_\ell-1\}$;
\item[(iii)] $\CA''$ is free with $\exp\CA'' = \{ b_1, \ldots , b_{\ell -1}\}$.
\end{itemize}
\end{theorem}

The following is Terao's celebrated  
\emph{Factorization Theorem} for  
free arrangements.

\begin{theorem}
[{\cite[Thm.\ 4.137]{orlikterao:arrangements}}]
\label{thm:factorization}
If $\CA$ is free with $\exp\CA = \{ b_1, \ldots , b_\ell\}$, then 
\[
\chi(\CA,t) = \prod\limits_{i=1}^\ell (t - b_i).
\]
\end{theorem}

\subsection{Inductively Free Arrangements}
\label{ssect:indfree}

An iterative application of the 
addition part of Theorem \ref{thm:add-del} leads to the class of  
\emph{inductively free} arrangements.

\begin{definition}
[{\cite[Def.~4.53]{orlikterao:arrangements}}]
\label{def:indfree}
The class $\CIF$ of \emph{inductively free} arrangements 
is the smallest class of arrangements subject to
\begin{itemize}
\item[(i)] $\Phi_\ell \in \CIF$ for each $\ell \ge 0$;
\item[(ii)] if there exists a hyperplane $H_0 \in \CA$ such that both
$\CA'$ and $\CA''$ belong to $\CIF$, and $\exp \CA '' \subseteq \exp \CA'$, 
then $\CA$ also belongs to $\CIF$.
\end{itemize}
We denote by $\CIF_\ell$ the subclass of $\ell$-arrangements in $\CIF$.
\end{definition}

\subsection{Divisionally Free Arrangements}
\label{ssect:divfree}
First we recall the key result from \cite{abe:divfree}.

\begin{theorem}
[{\cite[Thm.~1.1]{abe:divfree}}]
\label{thm:div-free}
Let $\CA \ne \Phi_\ell$.
Suppose there is a hyperplane $H$ in $\CA$ such that the restriction 
$\CA^H$ is free and that 
$\chi(\CA^H,t)$ divides $\chi(\CA,t)$.
Then $\CA$ is free.
\end{theorem}
  
Theorem \ref{thm:div-free} can be viewed as a strengthening of the addition
part of Theorem \ref{thm:add-del}. An iterative  
application leads to the class $\CDF$.

\begin{definition}
[{\cite[Def.~1.5]{abe:divfree}}]
\label{def:divfree}
An arrangement $\CA$ is called 
\emph{divisionally free} if either $\ell \le 2$, $\CA = \Phi_\ell$, or
else there is a flag of subspaces $X_i$ of rank $i$ in $L(\CA)$
\[
X_0 = V \supset X_1 \supset X_2 \supset \cdots \supset X_{\ell -2}
\]
 so that $\chi(\CA^{X_i},t)$ divides $\chi(\CA^{X_{i-1}},t)$, for 
$i = 1, \ldots, \ell-2$.
Denote this class by $\CDF$.

We denote by $\CDF_\ell$ the subclass of $\ell$-arrangements in $\CDF$.
\end{definition}

In \cite[Thms.~1.3 and ~1.6]{abe:divfree}, Abe observed that $\CIF \subsetneq \CDF$
(the reflection arrangement of the complex reflection group $G_{31}$ is 
divisionally free but not inductively free), each $\CA$ in $\CDF$ is free and 
Terao's Conjecure \ref{conj:terao} is valid in $\CDF$.

\subsection{Additionally Free and Stair-Free Arrangements}
\label{ssect:divfree}
  
Using the addition part of Theorem \ref{thm:add-del}, it is natural to consider the following class.

\begin{definition}
[{\cite[Def.~1.6]{abe:sf}}]
\label{def:addfree}
An arrangement $\CA$ is called 
\emph{additionally free} if there is a filtration 
\[
\Phi_\ell = \CA_0 \subset \CA_1 \subset \ldots \subset \CA_n = \CA,
\]
of $\CA$, where each $\CA_i$ is free with $|\CA_i| = i$ and $|\CA| = n$.
Denote this class by $\CAF$.

We denote by $\CAF_\ell$ the subclass of $\ell$-arrangements in $\CAF$.
\end{definition}

The members of $\CAF$ are constructed by means of the
addition part of Theorem \ref{thm:add-del}. In particular, 
each member of $\CAF$ is free. 
Clearly, $\CIF \subseteq \CAF$.
In \cite[Thm.~1.8]{abe:sf}, Abe showed 
that Terao's conjecture is still valid within 
$\CAF$.

Combining the procedures of addition from Theorem \ref{thm:add-del}
and the construction of freeness from Theorem \ref{thm:div-free},
we obtain the following new natural class.

\begin{definition}
[{\cite[Def.~4.2]{abe:sf}}]
\label{def:sf}
An arrangement $\CA$ is called 
\emph{stair-free} if $\CA$ is build up from some empty arrangement
by consecutive applications of addition from Theorem \ref{thm:add-del}
or an extension along Theorem \ref{thm:div-free}.
Denote this class by $\CSF$.

We denote by $\CSF_\ell$ the subclass of $\ell$-arrangements in $\CSF$.
\end{definition}
 
 The significance of this new class $\CSF$ stems from the following result.

\begin{theorem}
[{\cite[Thm.~4.3]{abe:sf}}]
\label{thm:sf-free}
With the notation from above, we have 
\begin{itemize}
\item[(i)] every member of $\CSF$ is free;
\item[(ii)] $\CIF \subsetneq \CDF \cup \CAF \subseteq \CSF$; 
\item[(iii)] Terao's Conjecture \ref{conj:terao} is valid within $\CSF$.
\end{itemize}
\end{theorem}

Each of the classes of free, 
inductively free and divisionally free 
arrangements is compatible with 
the product construction for arrangements, 
cf.~\cite[Prop.\ 4.28]{orlikterao:arrangements}, 
\cite[Prop.~2.10]{hogeroehrle:inductivelyfree},
\cite[Prop.~2.9]{roehrle:divfree}.
Next we observe that this also holds for the classes $\CAF$
and $\CSF$.

\begin{proposition}
\label{prop:product-sffree}
Let $(\CA_1, V_1),  (\CA_2, V_2)$ be two arrangements.
Then  $\CA = (\CA_1 \times \CA_2, V_1 \oplus V_2)$ is 
stair-free (resp.~additionally free) if and only if both 
$\CA_1$ and $\CA_2$ are 
stair-free (resp.~additionally free).
\end{proposition}

\begin{proof}
First suppose that both $\CA_1$ and $\CA_2$ are stair-free.
We claim that then so is $\CA$.
We argue via induction on $|\CA|$. If both $\CA_1$ and $\CA_2$ are empty, 
there is nothing to show.
So suppose that $|\CA| \ge 1$ and 
that the claim holds for any 
product of stair-free arrangements with fewer  than $|\CA|$ 
hyperplanes.
Without loss of generality there exists an $H_1$ in $\CA_1$ 
such that either $\CA_1\setminus\{H_1\}$ 
or $\CA_1^{H_1}$  still belongs to $\CSF$.
Consequently, in the first instance
$(\CA_1\setminus\{H_1\}) \times \CA_2 \in \CSF$, 
by induction.
Moreover, as both 
$\CA_1$ and $\CA_1\setminus\{H_1\}$ are free, 
it follows from the strong form of the restriction part of 
Theorem \ref{thm:add-del} (\cite[Thm.\ 4.46]{orlikterao:arrangements})
that also the restriction $\CA_1^{H_1}$ is free and 
$\exp( \CA_1^{H_1}) \subset \exp( \CA_1\setminus\{H_1\})$.
Consequently, setting $H := H_1 \oplus V_2 \in \CA$ and using \eqref{eq:restrproduct} and   
\cite[Prop.\ 4.28]{orlikterao:arrangements}, we have
\[
\exp( \CA^{H}) =  \{\exp( \CA_1^{H_1}), \exp(\CA_2)\} \subset \{\exp( \CA_1\setminus\{H_1\}), \exp(\CA_2) \} = \exp( \CA\setminus\{H\}),
\]
and so by the addition part of Theorem \ref{thm:add-del}, 
$\CA$ also belongs to $\CSF$.

In the second instance when $\CA_1^{H_1}$  still belongs to $\CSF$, we have that
$\chi( \CA_1^{H_1}, t)$ divides $\chi( \CA_1, t)$.
For $H = H_1 \oplus V_2 \in \CA$, we see that 
$\CA^H = \CA_1^{H_1} \times \CA_2$ is a product of 
stair-free arrangements with $|\CA^H| < |\CA|$, by \eqref{eq:restrproduct}. 
So, by our induction hypothesis, 
$\CA^H$ is stair-free.
In addition, since 
$\chi( \CA^{H}, t) = \chi( \CA_1^{H_1}, t) \cdot \chi( \CA_2, t)$ divides 
$\chi( \CA_1, t) \cdot \chi( \CA_2, t) = \chi( \CA, t)$, 
cf.~\eqref{eq:chiproduct}, 
we infer that $\CA$ belongs to $\CSF$, by Theorem \ref{thm:div-free}.

Conversely, suppose that 
$\CA = \CA_1 \times \CA_2$ belongs to 
$\CSF$. We claim that then both $\CA_1$ and $\CA_2$ also belong to $\CSF$.
Again we argue by induction on $|\CA|$.
If both $\CA_1$ and $\CA_2$ are empty, 
there is nothing to show.
So suppose that $|\CA| \ge 1$ 
and that the claim holds for any 
product in $\CSF$ with fewer than $|\CA|$ hyperplanes.
Without loss of generality we may assume that there is an 
$H = H_1 \oplus V_2$ in $\CA$ such that either 
$ \CA\setminus\{H\} =  \CA_1\setminus\{H_1\} \times \CA_2$ 
or  $\CA^H$ belongs to $\CSF$
and $\chi( \CA^H, t)$ divides $\chi( \CA, t)$ in the second instance.
Thus in the first case, by our induction hypothesis, both  
$\CA_1\setminus\{H_1\}$ and $\CA_2$ also belong to $\CSF$.
Once again, by the
strong form of the restriction part of 
Theorem \ref{thm:add-del} (\cite[Thm.\ 4.46]{orlikterao:arrangements})
also the restriction $\CA^{H}$ is free and 
$\exp( \CA^{H}) \subset \exp( \CA\setminus\{H\})$.
Thus it follows from 
\cite[Prop.\ 4.28]{orlikterao:arrangements} that
\[
\{\exp( \CA_1^{H_1}), \exp(\CA_2)\} = \exp( \CA^{H}) \subset \exp( \CA\setminus\{H\}) = \{\exp( \CA_1\setminus\{H_1\}), \exp(\CA_2) \},
\]
and so $\exp( \CA_1^{H_1}) \subset \exp( \CA_1\setminus\{H_1\})$.
Therefore, by Theorem \ref{thm:add-del}, 
$\CA_1$ also belongs to $\CSF$.

Now consider the second case when $\CA^H$ belongs to $\CSF$
and $\chi( \CA^H, t)$ divides $\chi( \CA, t)$.
Since $|\CA^H| < |\CA|$ and  
$\CA^H = \CA_1^{H_1} \times \CA_2$, 
both $\CA_1^{H_1}$ and $\CA_2$ belong to $\CSF$, by our 
induction hypothesis.
Moreover,  since 
$\chi( \CA^{H}, t) = \chi( \CA_1^{H_1}, t) \cdot \chi( \CA_2, t)$ 
and 
$\chi( \CA, t) = \chi( \CA_1, t) \cdot \chi( \CA_2, t)$,
cf.~\eqref{eq:chiproduct}, 
it follows that 
$\chi( \CA_1^{H_1}, t)$ divides $\chi( \CA_1, t)$.
Therefore, also $\CA_1$ belongs to $\CSF$.

The statement for $\CAF$ of the proposition follows from just parts of the argument above.
\end{proof}

\begin{remark}
Clearly, $\CAF$ is closed under taking localizations. For, freeness is closed under taking localizations 
(\cite[Thm.\ 4.37]{orlikterao:arrangements}), 
so that a free chain of a member of $\CAF$ descends to a free chain of any 
of its localizations by removing 
redundant hyperplanes.
In contrast, by \cite[Ex.~2.16]{roehrle:divfree}, $\CDF$ is not closed under taking localizations, thus neither is $\CSF$.
\end{remark}

We close this section by discussing the reflection arrangements that belong to $\CSF$.

\begin{example}
Let $W$ be an irreducible unitary reflection group 
with $\CA(W)$ its reflection arrangement consisting of the hyperplanes associated with reflections in $W$,
see \cite[\S 6]{orlikterao:arrangements}.
Then $\CA(W)$ belongs to $\CSF$ if and only
if either $\CA(W)$ is inductively free or else $W = G_{31}$.
For, by Theorem \ref{thm:sf-free} and \cite[Thm.~1.6]{abe:divfree},
each inductively free reflection arrangement and also $\CA(G_{31})$ belongs to $\CSF$.
In contrast, none of the remaining irreducible, non-inductively free reflection arrangements $\CA(W)$ belongs to $\CSF$.
Indeed, for any such  $\CA(W)$ and any choice of hyperplane, $\CA(W)'$ fails to be free. Moreover,  
by \cite[Cor.~1.3, Cor.~2.18]{hogeroehrle:inductivelyfree}, 
$\exp \CA(W)'' \not\subseteq \exp \CA(W)$, 
so that $\chi( \CA(W)'', t)$ does not divide $\chi( \CA(W), t)$. So $\CA(W) \notin \CSF$, by Definition \ref{def:sf}.
\end{example}

\section{Proof of Theorem \ref{thm:abe-conj4.4}}
\label{sec:proofs}

\subsection{}
Observe that in dimension $3$ we have 
$\CSF_3 = \CDF_3 \cup \CAF_3 = \CDF_3  = \CIF_3$. For, since every rank $2$ arrangement is already inductively free,
the result follows from \cite[Thm.\ 4.46, Prop.\ 4.52]{orlikterao:arrangements}. Consequently, examples to demonstrate 
the statements claimed in Theorem \ref{thm:abe-conj4.4} can only occur in dimension at least $4$.

\subsection{}
\label{sub:B}
We first consider the inductively free arrangement $\CA$ 
of rank $7$ consisting of $32$ hyperplanes, 
with induction table given in Table \ref{indtable:A} below.

\begin{table}[h] \tiny
\begin{tabular}{lll}  \hline
  $\exp \CA'$ & $\alpha_H$ & $\exp \CA''$\\
\hline\hline
$0, 0, 0, 0, 0, 0, 0$ & $x_1 + x_2 + 2x_3 + 2x_4+ 2x_5 + x_6 + x_7$ & $0, 0, 0, 0, 0, 0$\\
$0, 0, 0, 0, 0, 0, 1$ & $x_1 + 2x_2 + 2x_3 + 3x_4 + 2x_5 + x_6 + x_7$ & $0, 0, 0, 0, 0, 1$\\
$0, 0, 0, 0, 0, 1, 1$ & $x_2 + x_4$ & $0, 0, 0, 0, 0, 1$\\
$0, 0, 0, 0, 0, 1, 2$ & $x_4$ & $0, 0, 0, 0, 1, 2$\\
$0, 0, 0, 0, 1, 1, 2$ & $x_2$ & $0, 0, 0, 0, 1, 2$\\
$0, 0, 0, 0, 1, 2, 2$ & $x_1 + x_2 + 2x_3 + 3x_4 + 2x_5 + x_6 + x_7$ & $0, 0, 0, 0, 1, 2$\\
$0, 0, 0, 0, 1, 2, 3$ & $x_1 + 2x_2 + 2x_3 + 3x_4 + 3x_5 + 2x_6 + x_7$ & $0, 0, 0, 1, 2, 3$\\
$0, 0, 0, 1, 1, 2, 3$ & $x_1 + x_2 + 2x_3 + 3x_4 + 3x_5 + 2x_6 + x_7$ & $0, 0, 0, 1, 2, 3$\\
$0, 0, 0, 1, 2, 2, 3$ & $x_5 + x_6$ & $0, 0, 0, 1, 2, 3$\\
$0, 0, 0, 1, 2, 3, 3$ & $x_1 + 2x_2 + 2x_3 + 4x_4 + 3x_5 + 2x_6 + x_7$ & $0, 0, 0, 1, 3, 3$\\
$0, 0, 0, 1, 3, 3, 3$ & $x_4 + x_5 + x_6$ & $0, 0, 0, 1, 3, 3$\\
$0, 0, 0, 1, 3, 3, 4$ & $x_2 + x_4 + x_5 + x_6$ & $0, 0, 0, 1, 3, 3$\\
$0, 0, 0, 1, 3, 3, 5$ & $x_3 + x_4 + x_5$ & $0, 0, 1, 3, 3, 5$\\
$0, 0, 1, 1, 3, 3, 5$ & $x_2 + x_3 + 2x_4 + x_5$ & $0, 0, 1, 3, 3, 5$\\
$0, 0, 1, 2, 3, 3, 5$ & $x_2 + x_3 + 2x_4 + 2x_5 + x_6$ & $0, 0, 1, 3, 3, 5$\\
$0, 0, 1, 3, 3, 3, 5$ & $x_2 + x_3 + x_4 + x_5$ & $0, 0, 1, 3, 3, 5$\\
$0, 0, 1, 3, 3, 4, 5$ & $x_1 + x_2 + x_3 + 2x_4 + 2x_5 + 2x_6 + x_7$ & $0, 0, 1, 3, 4, 5$\\
$0, 0, 1, 3, 4, 4, 5$ & $x_2 + x_3 + 2x_4 + x_5 + x_6$ & $0, 1, 3, 4, 4, 5$\\
$0, 1, 1, 3, 4, 4, 5$ & $x_1 + x_2 + x_3 + 2x_4 + 2x_5 + x_6 + x_7$ & $0, 1, 3, 4, 4, 5$\\
$0, 1, 2, 3, 4, 4, 5$ & $x_3 + x_4 + x_5 + x_6$ & $0, 1, 3, 4, 4, 5$\\
$0, 1, 3, 3, 4, 4, 5$ & $x_2 + x_3 + x_4 + x_5 + x_6$ & $0, 1, 3, 4, 4, 5$\\
$0, 1, 3, 4, 4, 4, 5$ & $x_1 + 2x_2 + 3x_3 + 4x_4 + 3x_5 + 2x_6 + x_7$ & $0, 1, 4, 4, 4, 5$\\
$0, 1, 4, 4, 4, 4, 5$ & $x_1 + x_2 + x_3 + 2x_4 + 2x_5 + x_6$ & $1, 4, 4, 4, 4, 5$\\
$1, 1, 4, 4, 4, 4, 5$ & $x_2 + x_3 + 2x_4 + x_5 + x_6 + x_7$ & $1, 4, 4, 4, 4, 5$\\
$1, 2, 4, 4, 4, 4, 5$ & $x_1 + x_2 + x_3 + x_4 + x_5$ & $1, 4, 4, 4, 4, 5$\\
$1, 3, 4, 4, 4, 4, 5$ & $x_1 + x_3 + x_4 + x_5$ & $1, 4, 4, 4, 4, 5$\\
$1, 4, 4, 4, 4, 4, 5$ & $x_3$ & $1, 4, 4, 4, 4, 5$\\
$1, 4, 4, 4, 4, 5, 5$ & $x_1 + x_2 + x_3 + 2x_4 + x_5 + x_6 + x_7$ & $1, 4, 4, 4, 5, 5$\\
$1, 4, 4, 4, 5, 5, 5$ & $x_6$ & $1, 4, 4, 5, 5, 5$\\
$1, 4, 4, 5, 5, 5, 5$ & $x_1 + x_2 + 2x_3 + 2x_4 + 2x_5 + x_6$ & $1, 4, 5, 5, 5, 5$\\
$1, 4, 5, 5, 5, 5, 5$ & $x_3 + x_4$ & $1, 4, 5, 5, 5, 5$\\
$1, 4, 5, 5, 5, 5, 6$ & $x_1$ & $1, 5, 5, 5, 5, 6$\\
$1, 5, 5, 5, 5, 5, 6$ \\
\hline                                
\end{tabular}\\
\bigskip
\caption{Induction table for the subarrangement $\CA$ of $\CA(E_7)$.} 
\label{indtable:A} 
\end{table}

Here $\CA$ is realized as a subarrangement
of the Weyl arrangement $\CA(E_7)$ of the Weyl group of type $E_7$. 
The $x_1, \ldots, x_7$ represent the simple roots 
according to the labeling in 
\cite[Planche VI]{bourbaki:groupes}.
One checks that the resulting arrangement is inductively free with exponents $\exp\CA = \{1, 5, 5, 5, 5, 5, 6\}$. 
Of course, this entails checking inductive freeness of all rank $6$ restrictions in Table \ref{indtable:A} and again 
their restrictions, etc. 
In particular, if we remove the last hyperplane in the inductive chain, $\ker (x_1)$, then $\CA'$ is still inductively free with  
$\exp\CA' = \{1, 4, 5, 5, 5, 5,  6\}$.

However, if instead we remove the penultimate hyperplane from $\CA$ in the chain below, $\ker(x_3+x_4)$, 
then the resulting arrangement, say $\CB$, while still additionally free with exponents $\exp\CB = \{1, 5, 5,  5, 5, 5, 5\}$,
is no longer inductively free, as no restriction $\CB''$ 
with matching exponents  $\{1, 5, 5,  5, 5, 5\}$
is inductively free, see Table \ref{indtable:B}. 
Indeed, up to isomorphism there are only two restrictions $\CB''$ with fitting exponents 
$\exp\CB'' = \{1, 5, 5,  5, 5, 5\}$. Both of them are again additionally free but neither of them is inductively free.
If we consider further all possible restrictions of these two types of restrictions $\CB''$ with matching set of exponents 
$\{1, 5, 5,  5, 5\}$, then there is only one such further restriction  
of rank $5$ up to isomorphism. This is the arrangement $\CD$ which we are going to  
examine in \S \ref{sec:C}. There we show that $\CD$ is not inductively free which in turn shows that 
$\CB$ is not inductively free either.
This in particular then implies that  $\CB \in \CAF\setminus \CIF$. 

\begin{table}[h]
\begin{tabular}{lll}  \hline
  $\exp \CA'$ & $\alpha_H$ & $\exp \CA''$\\
\hline\hline
\dots & \dots & \dots \\
$1, 4, 4, 4, 5, 5, 5$ & $x_6$ & $1, 4, 4, 5, 5, 5$\\
$1, 4, 4, 5, 5, 5, 5$ & $x_1 + x_2 + 2x_3 + 2x_4 + 2x_5 + x_6$ & $1, 4, 5, 5, 5, 5$\\
$1, 4, 5, 5, 5, 5, 5$ & $x_1$ & $1, 5, 5, 5, 5, 5$\\
$1, 5, 5, 5, 5, 5, 5$  \\
\hline                                
\end{tabular}\\
\bigskip
\caption{Chain of hyperplanes for $\CB$ in $\CAF$.} 
\label{indtable:B} 
\end{table}

If we further remove $\ker (x_1)$ from $\CB$, the resulting arrangement $\CB'$ is of course inductively free again,
by Table \ref{indtable:A}, as it coincides with $\CA\setminus\{\ker (x_1), \ker(x_3+x_4)\}$. So
the non-inductively free arrangement   
 $\CB$ is tightly sandwiched between the  inductively free arrangements $\CB'$ and $\CA$.

\subsection{}
\label{sec:C}
Next, we consider the 
restriction $\CC :=\CA^Z$ of $\CA$, where $Z$ is the intersection of the hyperplanes
$H_1 : = \ker(x_1)$ and $H': = \ker(x_1 + x_2 + 2x_3 + 2x_4 + 2x_5 + x_6)$.
Then $\CC$ is a subarrangement of the restriction $\CA(E_7)^Z$ of $\CA(E_7)$ which is of type $(E_7, A_1^2)$.

One checks that $\CC$ is again 
inductively free with exponents $\exp\CC = \{1, 5, 5, 5, 6\}$. 
An induction table for $\CC$ is given in Table \ref{indtable:C}. 
In particular, if we remove the last hyperplane in the inductive chain, $\ker (x_4)$, then $\CC'$ is still inductively free with  $\exp\CC' = \{1, 4, 5, 5, 6\}$.

\begin{table}[h]
\begin{tabular}{lll}  \hline
  $\exp \CA'$ & $\alpha_H$ & $\exp \CA''$\\
\hline\hline
$0, 0, 0, 0, 0$ {\hglue 5pt}	   &    $x_2$ 			             & $0,0,0,0$ \\
$0, 0, 0, 0, 1$ 	   &    $x_1+x_3-x_5$				             & $0,0,0,1$ \\
$0, 0, 0, 1, 1$ 	   &    $2x_1+x_2+x_3$			            & $0, 0,1,1$ \\    
$0, 0, 1, 1, 1$ 	   &    $2x_1+x_2+2x_3+x_4-x_5$		   & $0,1,1,1$ \\
$0, 1, 1, 1, 1$ 	   &    $x_5$					             & $1, 1,1,1$ \\
$1, 1, 1, 1, 1$ 	   &    $x_1+x_3$					             & $1,1,1, 1$ \\
$1, 1, 1, 1, 2$ 	   &    $x_2+x_5$					             & $1, 1,1,2$ \\    
$1, 1, 1, 2, 2$ 	   &    $2x_1+x_2+2x_3+x_4$				             & $1, 1,2,2$ \\    
$1, 1, 2, 2, 2$ 	   &    $2x_1+x_3-x_5$					             & $1, 2,2,2$ \\
$1, 2, 2, 2, 2$ 	   &    $2x_1+2x_2+2x_3+x_4$					             & $1, 2,2,2$ \\        
$1, 2, 2, 2, 3$ 	   &    $x_2+x_3+x_4$					             & $1, 2,2,3$ \\
$1, 2, 2, 3, 3$ 	   &    $x_1+x_2+x_3+x_4$                        &       $1, 2,3,3$  \\
$1, 2, 3, 3, 3$ 	   &    $x_3+x_4$                         &        $1, 3,3,3$ \\
$1, 3, 3, 3, 3$ 	   &    $x_1+x_2+x_3$                         &  $1, 3,3,3$       \\
$1, 3, 3, 3, 4$ 	   &    $x_1$                         &        $1, 3,3,4$ \\
$1, 3, 3, 4, 4$ 	   &    $x_1+x_3+x_4$                         &  $1, 3,4,4$       \\
$1, 3, 4, 4, 4$ 	   &    $2x_1+x_2+x_3-x_5$                         & $1, 4,4,4$        \\
$1, 4, 4, 4, 4$ 	   &    $x_2+x_3+x_4+x_5$                         & $1, 4,4,4$        \\
$1, 4, 4, 4, 5$ 	   &    $x_1-x_5$                         & $1, 4,4,5$        \\
$1, 4, 4, 5, 5$ 	   &    $x_1-x_4-x_5$                         & $1, 4,5,5$        \\
$1, 4, 5, 5, 5$ 	   &    $x_1+x_2$                         & $1, 4,5,5$        \\
$1, 4, 5, 5, 6$ 	   &    $x_4$                         & $1, 5,5,6$        \\
$1, 5, 5, 5, 6$ \\
\hline                                
\end{tabular}\\
\bigskip
\caption{Induction table for the rank $5$ arrangement $\CC$.} 
\label{indtable:C} 
\end{table}

However, if instead we remove the penultimate hyperplane from $\CC$ in the chain in Table \ref{indtable:C}, 
$\ker(x_1+x_2)$, 
then the resulting arrangement, say $\CD$, while still additionally free with exponents $\exp\CD = \{1, 5, 5, 5, 5\}$,
is no longer inductively free, as no restriction $\CD''$ 
with matching exponents $\{1, 5, 5, 5\}$
is inductively free, see Table \ref{indtable:D}.
Up to isomorphism there is only one restriction $\CD'' \cong \CD^{\ker x_4}$ with $\exp\CD'' = \{1, 5, 5, 5\}$. While this 
restriction is necessarily free, 
it is no longer additionally free (and so clearly not inductively free). 
For any choice of hyperplane in $\CD''$, the resulting deletion even if free
does not have matching exponents $\{1, 4, 5, 5\}$.
In particular, we have $\CD \in \CAF\setminus \CIF$. 
This in particular proves Theorem \ref{thm:abe-conj4.4}(i). 
In addition this also shows that  
$\CAF$ is not closed under taking restrictions.

If we further remove $\ker (x_4)$ from $\CD$, the resulting arrangement $\CD'$ is of course inductively free again,
by Table \ref{indtable:C}, as it coincides with $\CC\setminus\{\ker (x_4), \ker(x_1+x_2)\}$. So
the non-inductively free arrangement   
 $\CD$ is sandwiched between the  inductively free arrangements $\CD'$ and $\CC$.

\begin{table}[h]
\begin{tabular}{lll}  \hline
  $\exp \CA'$ & $\alpha_H$ & $\exp \CA''$\\
\hline\hline
\dots & \dots & \dots \\
$1, 4, 4, 5, 5$ 	   &    $x_1-x_4-x_5$                         & $1, 4,5,5$        \\
$1, 4, 5, 5, 5$ 	   &    $x_4$                         & $1, 5,5,5$        \\
$1, 5, 5, 5, 5$ \\
\hline                                
\end{tabular}\\
\bigskip
\caption{Chain of hyperplanes for $\CD$ in $\CAF$.} 
\label{indtable:D} 
\end{table}

Moreover, one checks that $\CD$ is not divisionally free. For, there is no restriction of 
$\CD''$ with exponents $\{1, 5, 5\}$ and so the characteristic polynomial of any such restriction 
does not divide the characteristic polynomial of  $\CD''$.
In particular, this proves Theorem \ref{thm:abe-conj4.4}(ii).

As a subarrangement of $\CC$, also $\CD$ is a subarrangement of the restriction of 
$\CA(E_7)$ of type $(E_7,A_1^2)$.
Explicitly, $\CD$ is obtained from the 
arrangement  $\CB$ as the restriction $\CD = \CB^X$, where $X := \ker(x_1) \cap \ker(x_6)$.
The properties of $\CD$ we have established imply that $\CB$ above 
also satisfies the conditions in  Theorem \ref{thm:abe-conj4.4}(i) and (ii).
These are the only examples known to us which satisfy these properties.

We further observe that if we label the last three hyperplanes in the chain for $\CB$ in Table \ref{indtable:B} by 
$H_1 : = \ker(x_1)$, $H': = \ker(x_1 + x_2 + 2x_3 + 2x_4 + 2x_5 + x_6)$,  and $H_6 : = \ker(x_6)$,
then these hyperplanes $H$ are precisely the ones so that  the  restriction $\CB^H$ has got the 
required exponents $\{1, 5, 5,  5, 5, 5\}$ to satisfy the deletion part of Theorem \ref{thm:add-del} for $\CB$. Further, for
 $Y := H_1 \cap H' \cap H_6$ we have
$\CB_Y = \{H_1, H',  H_6\}$ and $\CB^Y = \CD''$.
This implies that also $\CB$ fails to be divisionally free. For, a divisional chain as in Definition \ref{def:divfree} 
necessarily has to pass through a restriction isomorphic to $\CD''$.

\subsection{}
Thanks to \cite[Thm.~1.6]{abe:divfree}, the reflection arrangement $\CA(G_{31})$ of the complex reflection group 
$G_{31}$ is divisionally free but it is not additionally free, see the proof of \cite[Lem.\ 3.5]{hogeroehrle:inductivelyfree}.  
Thus $\CA(G_{31})$  belongs to $\CDF$ but not to $\CAF$.
Since $\CD$ above belongs to $\CAF$ but not to $\CDF$, 
it follows from 
Proposition \ref{prop:product-sffree} and \cite[Prop.~2.9]{roehrle:divfree}
that $$\CE:= \CD \times \CA(G_{31})$$ neither belongs to $\CDF$, nor to $\CAF$, but 
at the same time $\CE$ is still stair-free, i.e.
\[\CE \in \CSF\setminus (\CDF \cup \CAF)\] and so Theorem \ref{thm:abe-conj4.4}(iii)
follows. It would be interesting to know of an irreducible example in $\CSF\setminus (\CDF \cup \CAF)$. 

\subsection{}
\label{rem:computations}
The facts that $\CA$ and $\CC$ above  are  
inductively free and that both
$\CB$ and $\CD$ are still   
additionally free  
were checked by computational means.
Likewise the fact that $\CD$ is 
not divisionally free was checked with the aid of a computer.

\section{Non-free triples of arrangements}
\label{sec:proof2}

In this section we discuss counterexamples to 
another conjecture of Abe, \cite[Conj.~3.5(2)]{abe:sf}. 
Specifically, 
here we provide an example of a triple of arrangements $(\CA, \CA', \CA'')$ with the property that 
none of them is free but each of their characteristic polynomials factors over $\BBZ$ and 
$\chi(\CA'', t)$ divides both $\chi(\CA', t)$ and $\chi( \CA, t)$ so that the latter two 
polynomials share all but one root.

\begin{example}
\label{ex:divnotfree}
Let $w, x, y, z$ be indeterminates over $\BBQ$ and let 
$\CA$ be the arrangement in $\BBQ^4$ with 11 hyperplanes given by
\[
Q(\CA) = wxyz(x+y)(x+z)(x-z)(y-z)(y+z)(x+y-z)(w+x-y).
\]
It is easy to check that for $H = \ker(x+y-z)$, we have
\begin{align*}
\chi( \CA, t) & =  (t-1)(t-3)^2(t-4),\\ 
\chi( \CA', t) & =  (t-1)(t-3)^3, \text{ and } \\
\chi( \CA'', t) & =  (t-1)(t-3)^2.
\end{align*}
Although the factorization over $\BBZ$ of each of  
these polynomials is consistent with
Terao's Factorization Theorem \ref{thm:factorization}, none
of the arrangements in the triple $(\CA, \CA', \CA'')$ is free.
\end{example}

In the following we provide a general construction to generate counterexamples to 
\cite[Conj.~3.5(2)]{abe:sf} with an arbitrary number of hyperplanes in any dimension at least $3$.

\begin{example}
\label{ex:divnotfree2}
Let $\CB$ be a fixed non-free arrangement in dimension $\ell \ge 3$ over $\BBQ$ with the property that 
its characteristic polynomial factors over $\BBZ$, e.g.~take the arrangement $\CA''$ from Example \ref{ex:divnotfree}.
Without loss we may assume that $\ker x \in \CB$, where $x$ is a coordinate of $\BBQ^{\ell}$.
Now view $\CB$ as an arrangement in $\BBQ^{\ell+1}$ and let  $z$ be the new coordinate.
Fix an integer $m\ge 0$ and define $\CA$  in $\BBQ^{\ell+1}$  by adding the hyperplanes 
$\ker(z), \ker(x-z), \ker(2x-z), \ldots, \ker(mx-z)$ to $\CB$. 
Consider the triple $(\CA, \CA', \CA'')$ with respect to 
$\ker(mx-z)$. Then we have $\CA'' \cong \CB$. By 
induction on $m$ and 
the fact that $\chi( \CA, t) = \chi( \CA', t) - \chi( \CA'', t)$ (\cite[Cor.~2.57]{orlikterao:arrangements}), we obtain 
\begin{align*}
\chi( \CA, t) & =  \chi( \CB, t)(t-m-1),\\ 
\chi( \CA', t) & =  \chi( \CB, t)(t-m), \text{ and } \\
\chi( \CA'', t) & =  \chi( \CB, t).
\end{align*}
Still none
of the arrangements in the triple $(\CA, \CA', \CA'')$ is free. For, 
the localization of $\CA$ at the center of $\CB$ in $L(\CA)$ is isomorphic to $\CB$ and thus is not free,
thus neither is $\CA$, by \cite[Thm.~4.37]{orlikterao:arrangements}; 
likewise for $\CA'$. 
\end{example}


\bigskip 

{\bf Acknowledgments}:  The research of this work was supported by 
DFG-grant RO 1072/16-1.


\bigskip

\bibliographystyle{amsalpha}

\begin{thebibliography}{AHR14}

\bibitem[Abe16]{abe:divfree}
T.~Abe,
\emph{Divisionally free arrangements of hyperplanes},
Inventiones Math. \textbf{204(1)}, (2016), 317--346.

\bibitem[Abe17]{abe:deletion}
\bysame,
\emph{Deletion theorem and combinatorics of hyperplane arrangements}.
Math. Ann., to appear. \url{arXiv:1709.05842}.

\bibitem[Abe18]{abe:sf}
\bysame,
\emph{Addition-deletion theorem for free hyperplane arrangements and combinatorics},  
\url{arXiv:1811.03780}

\bibitem[AHR14]{amendhogeroehrle:indfree}
N.~Amend, T.~Hoge and G.~R\"ohrle, 
\emph{On inductively free restrictions of reflection arrangements},
J. Algebra \textbf{418} (2014), 197--212.

\bibitem[Bou68]{bourbaki:groupes} N.~Bourbaki, \emph{\'{E}l\'ements de
    math\'ematique. {G}roupes et alg\`ebres de {L}ie. {C}hapitre {IV}-{VI}},
  Actualit\'es Scientifiques et Industrielles, No. 1337, Hermann, Paris,
  1968.

\bibitem[HR15]{hogeroehrle:inductivelyfree} 
T.~Hoge and G.~R\"ohrle, 
\emph{On inductively free reflection arrangements}, J.~Reine u.~Angew.~Math.
\textbf{701} (2015), 205--220. 

\bibitem[OS82]{orliksolomon:unitaryreflectiongroups}
P.~Orlik and L.~Solomon,
\emph{Arrangements Defined by Unitary Reflection Groups},
Math. Ann. \textbf{261}, (1982), 339--357.

\bibitem[OT92]{orlikterao:arrangements} P.~Orlik and H.~Terao,
  \emph{Arrangements of hyperplanes}, Springer-Verlag, 1992.
  
\bibitem[R\"o18]{roehrle:divfree}
G.~R\"ohrle, 
\emph{Divisionally free restrictions of reflection arrangements}. S\'em. Lothar. Combin. \textbf{77} ([2016-2018]), 
Art. B77e, 8 pp.  

\bibitem[Ter80]{terao:freeI} H.~Terao, \emph{Arrangements of hyperplanes and
    their freeness I, II}, J.~Fac.~Sci.~Univ.~Tokyo \textbf{27} (1980),
  293--320.

\end{thebibliography}

\newcommand{\etalchar}[1]{$^{#1}$}
\providecommand{\bysame}{\leavevmode\hbox to3em{\hrulefill}\thinspace}
\providecommand{\MR}{\relax\ifhmode\unskip\space\fi MR }
\providecommand{\MRhref}[2]{%
  \href{http://www.ams.org/mathscinet-getitem?mr=#1}{#2} }
\providecommand{\href}[2]{#2}


\end{document}